\documentclass{amsart}[12pt font]
\author{Tamanna Chatterjee}
\usepackage[foot]{amsaddr}
\address{University of Notre Dame}
\email{tchatter@nd.edu}
\date{\today}
\title{ Fourier transform on graded Lie algebras}
\usepackage{amssymb}
\usepackage{amsmath}
\usepackage{calligra}
\usepackage{amsfonts}
\usepackage{mathrsfs}
\usepackage{mathtools}
\usepackage{tikz-cd}
\usepackage{pgfkeys,pgfopts,xcolor}
\usepackage{ytableau}
\usepackage{placeins}
\usetikzlibrary{matrix}
\usepackage{geometry,mathtools}
\theoremstyle{plain}
\newtheorem{theorem}{Theorem}[section]
\newtheorem{cor}[theorem]{Corollary}
\newtheorem{lemma}[theorem]{Lemma}

\theoremstyle{def}
\newtheorem{remark}[theorem]{Remark}
\newtheorem{conjecture}[theorem]{Conjecture}
\newtheorem{assume}[theorem]{Assumption}
\newtheorem{definition}[theorem]{Definition}

\newcommand{\mf}{\mathfrak}
\newcommand{\mc}{\mathcal}
\newcommand{\mb}{\mathbb}

\newcommand{\ms}{\mathscr}

\newcommand{\eal}{\end{align*}}

\newcommand{\bal}{\begin{align*}}
\DeclareMathOperator{\p}{Perv}
\DeclareMathOperator{\f}{For}

\DeclareMathOperator{\rhm}{R\mathscr{H}\text{\kern -3pt{\calligra\large om}}\, }
\DeclareMathOperator{\lie}{Lie}
\DeclareMathOperator{\h}{H}

\DeclareMathOperator{\cu}{cusp}

\DeclareMathOperator{\i'}{\mathscr{I}}

\DeclareMathOperator{\E}{\mathcal{E}}

\DeclareMathOperator{\ind}{^{\mathfrak{g},\mathfrak{p}}\mathcal{I}}
\DeclareMathOperator{\inn}{Ind}
\DeclareMathOperator{\mg}{\mathfrak{g}}
\DeclareMathOperator{\rnn}{Res}
\DeclareMathOperator{\res}{^ {\mathfrak{g},\mathfrak{p}}\mathcal{R}}

\DeclareMathOperator{\pno}{\mathfrak{p}_{n,\Omega}}

\DeclareMathOperator{\tL}{\tilde{L}}
\DeclareMathOperator{\tl}{\tilde{\mf{l}}}

\DeclareMathOperator{\tM}{\tilde{M}}
\DeclareMathOperator{\phign}{\Phi_{\mf{g}_n}}
\DeclareMathOperator{\tm}{\tilde{\mf{m}}}
\DeclareMathOperator{\pr}{pr}
\begin{document}
\maketitle

\begin{abstract}

In this paper, we extend Lusztig's result on the Fourier transform for graded Lie algebras on sheaves with coefficients in a field $\Bbbk$  of positive characteristic, assuming Mautner's cleanness conjecture. We consider the centralizer $G_0$ of a fixed cocharacter $\chi$ in a connected, reductive, algebraic group $G$ and its action on the eigenspaces $\mf{g}_n$ of $\chi$. In particular, we consider $G_0$-equivariant sheaves on $\mf{g}_n$ and their behavior with respect to induction and restriction functors. Because our field $\Bbbk$ has positive characteristic, we consider parity sheaves, and many of our arguments require significant adaptation from arguments in \cite{Lu}. We prove in particular that Fourier transform takes parity sheaves to parity sheaves, and preserves the set of cuspidal sheaves.
\end{abstract}

\section{Introduction}

In this paper, we investigate the Fourier transform for graded Lie algebras on sheaves with coefficients in a field $\Bbbk$ of positive characteristic. The geometry of the $\mathbb{Z}$-graded Lie algebra has been explored by Lusztig \cite{Lu} in characteristic $0$. The primary aim of this study is to extend Lusztig's results related to the Fourier transform on graded Lie algebras to positive characteristic.

To enhance our understanding of the representation theory of the Weyl group, a central objective of the generalized Springer correspondence has been to establish a block decomposition for the collection of all pairs consisting of nilpotent orbits and irreducible local systems on those orbits. This topic has been examined in \cite{Lu3}, \cite{Lu4}, and \cite{Lu5} for characteristic $0$, and in \cite{AJHR2}, \cite{AJHR3}, and \cite{AJHR*} for positive characteristic. Following this motivation, a long-term goal of this project is to find a similar block decomposition in the graded setting for the positive characteristic. As for generalized Springer theory, Fourier transform has played an important role in \cite{Lu} to define the block decomposition in the graded setting for sheaf coefficients of characteristic $0$.

Let $G$ be a complex, connected, reductive algebraic group, and let $\mathfrak{g}$ denote its Lie algebra. We fix a cocharacter $\chi: \mathbb{C}^\times \to G$. The group $\mathbb{C}^\times$ acts on $\mathfrak{g}$ via the adjoint action, and the space $\mathfrak{g}_n$ represents the $n$-th weight space under this action, thereby defining a grading on $\mathfrak{g}$. The centralizer $G_0$ of $\chi(\mathbb{C}^\times)$ also acts on $\mathfrak{g}_n$ via the adjoint action. 

The $G_0$-equivariant derived category of sheaves with coefficients in a field $\Bbbk$, $D^b_{G_0}(\mf{g}_n,\Bbbk)$ has been studied in \cite{Lu} when the characteristic of $\Bbbk$ is $0$ and  in \cite{Ch} when the characteristic is positive. In positive characteristic, parity sheaves, introduced in \cite{parity} play  a similar role to the intersection cohomology complexes ($\mc{IC}$'s) in characteristic $0$.  For any pair $(\mc{O},\mc{L})$, where $\mc{O}$ is a $G_0$-orbit in $\mf{g}_n$ and $\mc{L}$ is a $G_0$-equivariant irreducible local system on $\mc{O}$, there exists at most one indecomposable parity sheaf $\mc{E}$ up to shift with the property that $\mc{E}|_{\mc{O}}=\mc{L}[\dim \mc{O}]$, and we write $\mc{E}(\mc{O},\mc{L})$ to denote $\mc{E}$.   We will denote the collection of all such pairs $(\mc{O},\mc{L})$ by $\ms{I}(\mf{g}_n)$. Similarly, on the nilpotent cone, for each pair $(C,\mc{F})$, where $C$ is a $G$-orbit in the nilpotent cone $\mc{N}$ and $\mc{F}$ is a $G$-equivariant irreducible local system on $C$, there exists at most one parity sheaf, denoted by $\mc{E}(C,\mc{F})$ with the property, $\mc{E}(C,\mc{F})|_{C}=\mc{F}[\dim C] $.
We will denote this collection of all pairs $(C,\mathcal{F})$ by $\ms{I}(G)$. The existence of parity sheaves on the nilpotent cone  for all $(C,\mc{F})\in \ms{I}(G)$ has been established in \cite{parity}, and the author extended this result in \cite{Ch} for each $(\mc{O},\mc{L})\in \ms{I}(\mf{g}_n)$, the indecomposable sheaf $\mc{E}(\mc{O},\mc{L})$ exists provided $\Bbbk$ satisfies the Assumption \ref{assume}.

For a parabolic subgroup $P$ with a Levi subgroup $L$ containing $\chi(\mathbb{C}^\times)$, we define in \cite{Ch} two functors relating  the derived categories $D^b_{L_0}(\mf{l}_n,\Bbbk)$ and $D^b_{G_0}(\mf{g}_n,\Bbbk)$,\[\inn^{\mf{g}}_{\mf{p}}: D^b_{L_0}(\mf{l}_n,\Bbbk) \to D^b_{G_0}(\mf{g}_n,\Bbbk),
\] 
\[ \rnn^{\mf{g}}_{\mf{p}}: D^b_{G_0}(\mf{g}_n,\Bbbk)\to D^b_{L_0}(\mf{l}_n,\Bbbk).
\] 

Similar to the nilpotent cone, these two functors, induction and restriction play crucial roles in the study of the equivariant derived category in the graded setting. In \cite{Ch} the author proved when the characteristic of $\Bbbk$ meets certain assumptions (Assumption \ref{assume}),  $\inn^{\mathfrak{g}}_{\mathfrak{p}}$ preserves the parity complexes. In this paper, we first prove that,

\begin{theorem}
	$\rnn^{\mf{g}}_{\mf{p}}$ sends parity complexes to parity complexes when the characteristic of $\Bbbk$ satisfies Assumption \ref{assume}.
\end{theorem} 

Fourier transform on the equivariant sheaves on $\mf{g}$ has been studied extensively in \cite{HK}, \cite{Lu4}, \cite{Mi} for characteristic $0$ and for  the positive characteristic  in \cite{Ju}, \cite{AHJR4}, \cite{AM}. 
It plays an important role in the study of Springer theory and generalized Springer correspondence.


Lusztig's work on the Fourier transform on $\mf{g}_n$ in characteristic $0$ established that the Fourier transform sends cuspidal pairs \cite[~ 4.4]{Lu} to cuspidal pairs and semisimple complexes to semisimple complexes. The goal of this paper is to prove the appropriate analogs of these results in the positive characteristic. Motivated by Mautner's cleanness conjecture and previous research in modular representation theory, both on the nilpotent cone and on graded Lie algebras, we make specific assumptions about the  characteristic of $\Bbbk$ (Assumption \ref{assume}). Under these assumptions, we prove the following theorem on $\mf{g}_n$.
\begin{theorem}
\begin{enumerate}
	\item The Fourier Sato transform sends the $\mc{IC}$ sheaf associated to a  cuspidal pair to a $\mc{IC}$ sheaf associated to a cuspidal pair.
	\item The Fourier Sato transform sends parity sheaves to parity sheaves.
\end{enumerate}
	 
\end{theorem}

The results in this paper are  dependent on the Mautner's Conjectures (Conjecture \ref{cleanness} and Conjecture \ref{cuparity}).  The author and P. Achar have work towards proving these conjectures. This work is not yet complete.

	\subsection*{Outline}
	In Section \ref{sec1}, we establish the necessary background, assumptions, and notations. In Section \ref{sec2}, we define $\rnn^\mathfrak{g}_{\mathfrak{p}}$ in the graded setting and demonstrate that it sends parity complexes to parity complexes. Section \ref{sec4} introduces the Fourier transform for graded Lie algebras, where we prove that the Fourier transform maps cuspidal pairs to cuspidal pairs and parity complexes to parity complexes. 

\subsection*{Acknowledgement} 

The author wishes to thank her PhD advisor, Pramod N. Achar, for suggesting the original thesis problem from which this work developed. The author is deeply grateful to her current postdoc mentor, Sam Evens, for reading and providing extensive, valuable comments on multiple drafts; his input was crucial in shaping the final version of this paper.

\tableofcontents

\section{Background}\label{sec1}
In this section, we recall some definitions, notation, and previous results.
Let $\Bbbk$ be a field of characteristic $l\geq 0$. We  consider  sheaves with coefficients in $\Bbbk$. The varieties we consider are defined over $\mathbb{C}$. Let $H$ be a linear algebraic group and $X$ be an $H$-variety.  We denote by $D^b_H(X,\Bbbk)$ or $D^b_H(X)$, the  $H$-equivariant derived category of constructible sheaves, which is defined in \cite{BL}. Let  $\p_H(X,\Bbbk)$ be the full subcategory of $H$-equivariant perverse $\Bbbk$-sheaves. For a $\Bbbk$-module $V$, let $\underbar{V}_X$ or $\underbar{V}$ denote the constant sheaf on $X$ with fiber $V$.  

Let  $G$  be a connected, reductive, algebraic group over 
$\mathbb{C}$ and let $\mathfrak{g}$ be the Lie algebra of $G$. 
We  fix a cocharacter $\chi:\mathbb{C}^\times \to G$ and define: \[G_0=\{g\in G| g\chi(t)=\chi(t)g, \forall{t}\in \mathbb{C}^\times\}.\]
 For $n\in \mathbb{Z}$, define: \[\mathfrak{g}_n=\{x\in \mathfrak{g}|\hspace{1mm} \mathrm{Ad}(\chi(t))x=t^nx, \forall{t}\in \mb{C}^\times\}.\] This defines a grading on $\mathfrak{g}$,
 \[\mathfrak{g}=\bigoplus_{n\in \mathbb{Z}}\mathfrak{g}_n.\]Clearly, $\mathfrak{g}_0=\lie(G_0)$ and
 $G_0$ acts on $\mathfrak{g}_n$. 	For $n\neq0$, $G_0$ acts on $\mathfrak{g}_n$ with only finitely many orbits.	In \cite{Ch} we studied the $G_0$-equivariant bounded derived category of sheaves, $D^b_{G_0}(\mf{g}_n, \Bbbk)$, with some restriction on the field characteristic of $\Bbbk$ which will be discussed later in this section.

Here we will review some notation from \cite{Ch}.
Recall that $\mf{sl}_2$ is the Lie algebra of $SL_2$ generated by:
 \[e=\begin{pmatrix}
0&1\\0&0\\	
\end{pmatrix}, h=\begin{pmatrix}
1&0\\0&-1\\	
\end{pmatrix}, f=\begin{pmatrix}
0&0\\1&0\\	
\end{pmatrix}.
\]
 Let $J_n=\{\phi:\mf{sl}_2 \to \mf{g}|\hspace{1mm}\phi(e)\in \mf{g}_n, \phi(f)\in \mf{g}_{-n},\phi(h)\in \mf{g}_0\}$. We have an action of $G_0$ on $J_n$ by $(g,\phi)\to \mathrm{Ad}(g)\circ \phi$. 

 	The map from the set of $G_0$-orbits on $J_n$ to the set of $G_0$-orbits on $\mf{g}_n$, defined by $\phi \to \phi(e)$, is a bijection \cite[\text{Prop} ~3.3]{Lu}.	
 \subsection{Induction and restriction}
 Induction and restriction are two important functors both for the nilpotent cone and for graded Lie algebras which allow us to go from the sheaves on a Levi subalgebra to sheaves $\mf{g}$ or from the sheaves on $\mf{g}$ to sheaves a Levi subalgebra. Let $P$ be a parabolic subgroup of $G$ and $L$ be a Levi factor in $P$ with $U$, the unipotent radical.   For parabolic induction and restriction on the nilpotent cone we use the following diagram:
  \[\mathcal{N}_L\xleftarrow{\pi_P}\mc{N}_L+\mathfrak{u}\xrightarrow{e_P}G\times^{P}(\mc{N}_L+\mathfrak{u})\xrightarrow{\mu_P}\mc{N}_G
  .\]
 Here $\mathfrak{u}=\lie(U)$, $\pi_P, e_P$ are the obvious maps and $\mu_P(g,x)=\mathrm{Ad}(g)x$. Let \[i_P=\mu_P\circ e_P:\mc{N}_L+\mf{u}\to \mc{N}_G.\]
 
 The parabolic restriction functor,
  \[\rnn^G_P: D^b_G(\mc{N}_G,\Bbbk) \to D^b_L(\mc{N}_L,\Bbbk)
  \]is defined by $\rnn^G_P(\mathcal{F})={\pi_P}_!i_P^*\f^{G}_L(\mc{F})$. Here \[\f^G_L: D^b_G(\mc{N}_G,\Bbbk) \to D^b_L(\mc{N}_G,\Bbbk)\] is the forgetful functor.
   Parabolic induction is a functor:
   \[\inn^G_P:  D^b_L(\mc{N}_L,\Bbbk) \to D^b_G(\mc{N}_G,\Bbbk),  \]
and is defined by
   $\inn^G_P(\mc{F}):={\mu_P}_!(e_P^*\f^G_P)^{-1}\pi_P^*(\mc{F})$. 
 Here  $(e^*\f^G_P)^{-1}: D^b_P(\mc{N}_L+\mf{u}) \to D^b_G( G\times^P(\mc{N}_L+\mf{u}))$ is the induction equivalence map \cite{Ac}.
 
 We assume that the Levi subgroup $L$ contains $\chi(\mb{C}^\times)$.
 Let $\mathfrak{p}, \mathfrak{l}, \mathfrak{u}$ be the Lie algebras of $P,L,U$ respectively. Then $\mathfrak{p}, \mathfrak{l}, \mathfrak{u}$ inherit a grading from $\mathfrak{g}$.
 
 For induction and restriction on graded Lie algebras we use the diagram below:
 
 \begin{equation}\label{eq2.1}
 	\mf{l}_n\xleftarrow{\pi}\mf{p}_n\xrightarrow{e}G_0\times^{P_0}\mf{p}_n\xrightarrow{\mu}\mf{g}_n
 \end{equation}

 All the maps here have the same meaning as in the diagram above. 
 The  restriction,
 
 \[\rnn^{\mf{g}}_{\mf{p}}: D^b_{G_0}(\mf{g}_n,\Bbbk) \to D^b_{L_0}(\mf{l}_n,\Bbbk)
  \]is defined by $\rnn^{\mf{g}}_{\mf{p}}(\mathcal{F})={\pi}_!i^*\f^{G_0}_{L_0}(\mc{F})$ with $i: \mf{p}_n \xhookrightarrow{}  \mf{g}_n$ being the inclusion map. Here 
   the parabolic induction comes from the same diagram above:
   \[\inn^{\mf{g}}_{\mf{p}}:  D^b_{L_0}(\mf{l}_n,\Bbbk) \to D^b_{G_0}(\mf{g}_n,\Bbbk)  \]
and is defined by,
   $\inn^{\mf{g}}_{\mf{p}}(\mc{F}):={\mu}_!(e^*\f^{G_0}_{P_0})^{-1}\pi^*(\mc{F})$, where 
 $(e^*\f^{G_0}_{P_0})^{-1}: D^b_{P_0}(\mf{p}_n) \to D^b_{G_0}(G_0\times^{P_0}\mf{p}_n) $ is the induction equivalence map. 
 \subsection{The sets $\mathscr{I}(G,\Bbbk)$, $\mathscr{I}(\mathfrak{g}_n,\Bbbk)$ and cuspidal pairs}\label{subsec2.1} 
 
 The set $\mathscr{I}(G,\Bbbk)$ is the set of pairs $(C,\mathcal{E})$ where $C \subset \mathscr{N}_G$ is a nilpotent $G$-orbit in $\mathfrak{g}$ and $\mathcal{E}$ is an irreducible $G$-equivariant $\Bbbk$-local system on $C$ (up to isomorphism).  The number of $G$-orbits in  $\mc{N}_G$ is finite and $G$-equivariant irreducible $\Bbbk$ local systems on $C$ are in bijection with the irreducible $\Bbbk$-representations of the  component group $A_G(x):=G^x/(G^x)^o$, where $x$ is in $C$, $G^x=\{g\in G|\hspace{2mm} \mathrm{Ad}(g)x=x\}$ and,  $(G^x)^\circ$ is the identity component of $G^x$.  Hence  the set $\mathscr{I}(G,\Bbbk)$ is finite. Sometimes when there is no confusion about the field of coefficients, then we will just write $\mathscr{I}(G)$.
 
 Let $\mathscr{I}(\mathfrak{g}_n,\Bbbk)$ or $\mathscr{I}(\mathfrak{g}_n)$ be the set of all pairs $(\mathcal{O},\mathcal{L})$ where $\mathcal{O}$ is a  $G_0$-orbit in $\mathfrak{g}_n$ and $\mathcal{L}$ is an irreducible, $G_0$-equivariant $\Bbbk$-local system on $\mathcal{O}$(up to isomorphism). By similar reasoning as for $\ms{I}(G)$, $\mathscr{I}(\mathfrak{g}_n)$ is finite.

 A simple object $\mathcal{F}$ in $\p_G(\mathscr{N}_G,\Bbbk)$ is called cuspidal if $\rnn^G_P(\mc{F})=0$, for every proper parabolic $P$.
	 A pair $(C,\mathcal{E})\in \mathscr{I}(G)$, is called cuspidal if the corresponding simple perverse sheaf $\mathcal{IC}(C,\mathcal{E})$ is cuspidal. We will denote the collection of all cuspidal pairs on $\mc{N}_G$ by $\ms{I}(G)^{\cu}$. 
	 There is a modular reduction map \cite[~ 2.3]{AJHR} from $\ms{I}(G,\mb{K})$ to $\ms{I}(G,\Bbbk)$, where $\mb{K}$ is a field of characteristic $0$. A pair $(C,\mc{E})\in \ms{I}(G,\Bbbk)$ is called $0$-cuspidal if it is in the image of a cuspidal pair in $\ms{I}(G,\mb{K})$ under the modular reduction map. We  denote this collection by $\ms{I}(G)^{0-\cu}$ and $\ms{I}(G)^{0-\cu}\subset \ms{I}(G)^{\cu}$ \cite[Lemma ~2.3]{AJHR}.
	  A brief discussion on cuspidal pairs and $0$-cuspidal pairs can also be found in \cite[Section ~2]{Ch}.

\begin{definition}
A pair $(\mathcal{O},\mathcal{L})\in \mathscr{I}(\mathfrak{g}_n)$ will be called  cuspidal if there exists a pair $(C,\mc{E})\in \ms{I}(G)^{0-\cu}$, such that $C\cap \mf{g}_n=\mc{O}$ and $\mc{L}=\mc{E}|_{\mc{O}}$.
  We will denote the set of all cuspidal pairs on $\mf{g}_n$ by 
$\mathscr{I}(\mathfrak{g}_n)^{\cu}$.
\end{definition}
 
 \begin{definition}
 	A pair $(C,\mathcal{E})\in \mathscr{I}(G)$ is called $l$-clean if the corresponding sheaf $\mathcal{IC}(C,\mathcal{E})$ has vanishing stalks on $\bar{C}-C$. Similarly, a pair $(\mathcal{O},\mathcal{L})\in \mathscr{I}(\mathfrak{g}_n)$ is called $l$-clean if the corresponding sheaf $\mathcal{IC}(\mathcal{O},\mathcal{L})$ has vanishing stalks on $\bar{\mathcal{O}}-\mathcal{O}$. 
 \end{definition}

 \begin{remark}
 	If $(\mc{O},\mc{L})$ is $\ell$-clean then $\mathcal{IC}(\mathcal{O},\mathcal{L})= \mathcal{E}(\mathcal{O},\mathcal{L}) $.
 \end{remark}

A discussion  on cleanness can be found in \cite[~ 2.2]{Ch}.

\subsection{Parity sheaves}
We start this subsection with the following assumptions:

\begin{assume}\label{assume}
 	\begin{enumerate}
 		\item The characteristic $l$ of $\Bbbk$ is a pretty good prime for $G$ \cite[~ 2.6]{Ch}.
 		\item The field $\Bbbk$ is big enough for $G$;
i.e, for every Levi subgroup $L$ of $G$ and pair $(C_L,\mathcal{E}_L)\in \mathscr{I}(L)$, the irreducible $L$-equivariant $\Bbbk$-local system $\mathcal{E}_L $ is absolutely irreducible.  
\end{enumerate}
 \end{assume}

Parity sheaves were first introduced by Juteau, Mautner, and Williamson \cite{parity}. They are  constructible complexes on a stratified space where the strata satisfy some cohomology vanishing properties. Once these conditions are satisfied then for any stratified space $X$ with stratum $X_\lambda$ and local system $\mc{L}$ on $X_\lambda$, there exists at most one indecomposable parity sheaf, denoted by $\mc{E}(X_\lambda, \mc{L})$ with $\mc{E}(X_\lambda, \mc{L})|_{X_{\lambda}}=\mc{L}[\dim X_{\lambda}]$. A detailed discussion of parity sheaves is in \cite{parity} and a summarized version can be found in \cite[Subsection ~1.5]{Ch}. The cohomology vanishing property for our case is the following theorem that has been proved in \cite{Ch}.

\begin{theorem}[\cite{Ch}, Theorem 29]\label{th2.3}
	Under the Assumption \ref{assume}, for any pair $(\mc{O},\mc{L})\in \ms{I}(\mf{g}_n)$, we have:
	\[\h^i_{G_0}(\mc{O},\mc{L})=0 \text{  for $i$ odd.}\]
\end{theorem}

The next theorem provides the cohomology vanishing condition for the nilpotent cone.

\begin{theorem}
 Under the Assumption \ref{assume}, for any pair $(C,\mc{F})\in \ms{I}(G)$ we have,
	\[\h^i_{G}(C,\mc{F})=0 \text{  for $i$ odd.}\]
\end{theorem}
 
 This theorem has been proved in \cite{parity}.
 Once we have these theorems we can talk about parity sheaves on both the nilpotent cone and on graded pieces $\mf{g}_n$. 
 But unlike $\mc{IC}$s sheaves, the existence of parity sheaves is not automatic. The existence for the  nilpotent cone has been discussed in \cite[~ 4.3]{parity}. 
 Under the assumptions on the field characteristic, one of the main theorems in \cite{Ch} is the following.

\begin{theorem}[\cite{Ch}, Theorem. ~26]\label{th2.6}
	 Any cuspidal pair $(\mc{O},\mc{L})\in \ms{I}(\mf{g}_n)^{\cu}$ is clean and therefor parity sheaves exist for cuspidals.
\end{theorem}
 The existence of parity sheaves for graded Lie algebras has been proved in \cite[Theorem. ~36]{Ch} with the Assumption \ref{assume}.

 The following conjectures were a part of a series of  Mautner's unpublished conjectures. These were known for some specific cases but unknown in general. The author and P. Achar have work towards proving these conjectures in \cite{AC}.

\begin{conjecture} \label{cleanness}
	Under the Assumption \ref{assume}, Mautner's cleanness conjecture is true, that is every $0$-cuspidal in $\ms{I}(G)$ is $l$-clean.
\end{conjecture}

 \begin{conjecture}\label{cuparity}
	For any parabolic $P$ and Levi subgroup $L\subset P$ with $(C,\mc{F})\in \ms{I}(L)^{\cu}$, $\inn^G_P\mc{E}(C,\mc{F})$ is a  parity complex. 
\end{conjecture}
 
 	The results in this paper are dependent on the Assumption \ref{assume}  and the Conjectures \ref{cleanness} and \ref{cuparity}.

\section{Restriction}\label{sec2}
In this section our main goal is to prove that restriction sends parity complexes to parity complexes in the graded setting. This extends results about restriction related to parity complexes on the nilpotent cone.

 In \cite{Lu} it was proved that when the characteristic of $\Bbbk$ is $0$ then the restriction sends simple perverse sheaves to semi-simple complexes. Our result has some similarity with Lusztig's arguments but the failure of the decomposition theorem in positive characteristic means our proof requires new techniques.

Let $P$ be a parabolic subgroup with Levi subgroup $L$ containing $\chi(\mb{C}^\times)$ and $(\mc{O}',\mc{L}')\in \i'(\mf{l}_n)^{\cu}$. Let $P'$ be another parabolic subgroup with Levi subgroup $L'$ containing $\chi(\mb{C}^\times)$.  In this section we will study $\rnn^{\mg}_{\mf{p}'}\inn^{\mg}_{\mf{p}}\mc{E}(\mc{O}',\mc{L}')$.

Recall the induction diagram for the cuspidal pair $(\mc{O}',\mc{L}')\in \i'(\mf{l}_n)^{\cu}$ from \cite[Lemma ~5.2]{Ch}, 
\[
\begin{tikzcd}
\mc{O}'  & \mc{O}'+\mf{u}_n \arrow[l, "\pi"'] \arrow[r,  "e"] & G_0\times^{P_0} (\mc{O}'+\mf{u}_n) \arrow[r,"\mu"] & \mf{g}_n
\end{tikzcd},
\]
where $\pi, \mu$ and $e$ are from (\ref{eq2.1}). The induced space, $G_0\times^{P_0}(\mc{O}'+\mf{u}_n)$ can be identified with 

\[ \{(gP_0,x)\in G_0/{P_0}\times \mf{g}_n|\mathrm{Ad}(g^{-1})x\in \pi^{-1}(\mc{O}')\}. 
\] Under this identification the map $\mu$  simply becomes the projection on $\mf{g}_n$. Now consider the following Cartesian diagram,

\[\begin{tikzcd}
	\mc{O}'  & \mc{O}'+\mf{u}_n \arrow[l, "\pi"'] \arrow[r, hook,  "e"] & G_0\times^{P_0} (\mc{O}'+\mf{u}_n) \arrow[r,"\mu"] & \mf{g}_n \\ & &\mu^{-1}(\mf{p}'_n) \arrow[u, "j"] \arrow[rd, "\sigma"] \arrow[r, "\mu"] & \mf{p}'_n \arrow[u, "i", hook'] \arrow[d, "\pi'"]\\ 
	&&& \mf{l}'_n
\end{tikzcd}.
\] Here $\mu^{-1}(\mf{p}'_n)=\{(gP_0,x)\in G_0/{P_0}\times \mf{p}'_n|\mathrm{Ad}(g^{-1})x\in \pi^{-1}(\mc{O}')\}$ and the diagonal arrow $\sigma$, defined by the above diagram sends $(gP_0,x) \to \pi'(x)$. Using the above diagram and Theorem \ref{th2.6}, we see that  
\begin{equation}\label{eq3.2'}
	\rnn^{\mg}_{\mf{p}'}\inn^{\mg}_{\mf{p}}\mc{E}(\mc{O}',\mc{L}')\cong \sigma_!j^*(e^*\f^{G_0}_{P_0})^{-1}\pi^*\mc{L}'[\dim \mc{O}'].
\end{equation}
Let $\Omega$ be a $(P'_0,P_0)$-double coset of $G_0$. We define the subvariety, \[\mf{p}_{n, \Omega}=\{(gP_0,x)\in \Omega/{P_0}\times \mf{p}'_n|\mathrm{Ad}(g^{-1})x\in \pi^{-1}(\mc{O}')\} \subset \mu^{-1}(\mf{p}'_n),\] and note that $\mu^{-1}(\mf{p}'_n)=\cup_{\Omega} \pno$. Define $\sigma_{\Omega}=: \sigma|_{\pno}$. 
\begin{definition}\label{def3.1}
\begin{enumerate}

\item We define $K_{\Omega}=\sigma_{{\Omega}_!} ((e^*\f^{G_0}_{P_0})^{-1}\pi^*\mc{L}'[\dim \mc{O}']|_{\pno}) $.

	\item We define $\Omega$ to be good if there exists $g_0\in \Omega$ so that, $g_0Pg_0^{-1}\cap P'$ contains a Levi subgroup of $g_0Pg_0^{-1}$. We will say $\Omega$ is bad if it is not good. 
	
\end{enumerate}
\end{definition}

\subsection{Good double cosets} 
In this subsection we assume $\Omega$ is good. Let  $g_0\in \Omega$ and $g=hg_0$ with $h\in P_0'$ so that $(gP_0,z)\in \pno$. Consider the parabolic subgroup $Q=g_0Pg_0^{-1}$ which contains $\chi(\mb{C}^\times)$. Therefore, 
\[\mf{p}_{n, \Omega}=\{(hQ_0,x)\in P'_0Q_0/{Q_0}\times \mf{p}'_n|\mathrm{Ad}(h^{-1})x\in \mathrm{Ad}(g_0)\pi^{-1}(\mc{O}')\}.\] Let $M$ be a Levi subgroup of $Q$ containing $\chi(\mb{C}^\times)$ and let $\pi_Q: \mf{q}_n\to \mf{m}_n$ be the projection,  and let  $\mf{v}$ be the nil-radical of $\mf{q}$.   The space $\mf{p}_{n, \Omega} $ can be again identified with,
\begin{equation}\label{eq3.2}
	\{(h(P'_0\cap Q_0),x)\in P'_0/{P'_0\cap Q_0}\times \mf{p}'_n|\mathrm{Ad}(h^{-1})x\in \mc{O}''+\mf{v}_n\},
\end{equation}
where $\mc{O}''$ is the translation of $\mc{O}'$ by $\mathrm{Ad}(g_0)$. The map $\sigma_\Omega$ becomes $(h(P'_0\cap Q_0),x) \to \pi'(x)$.

 Let $\tL$ and $\tM$ be two Levi subgroups of $P'$ and $Q$  respectively sharing a common maximal torus $T$ containing $\chi(\mb{C}^\times)$, and denote their Lie algebras  by $\tl$ and $\tm$ respectively. Under the projection map, $\pi': \mf{p}'\to \mf{l}'$, $\tl$ can be identified with $\mf{l}'$. Similarly $\tm$ can be identified with $\mf{m}$ by the projection $\pi_Q: \mf{q}\to \mf{m}$.

\begin{theorem}\label{th2.2}
For $\Omega$ good, there exists a parabolic subgroup  $P''$ of $L'$ with Levi subgroup $L''$ such that $\mc{O}''\subset \mf{l}_n''$ and $(\mc{O}'',\mc{F})\in \i'(\mf{l}''_n)^{\cu}$ with $K_{\Omega}\cong \inn^{\mf{l}'}_{\mf{p}''}\E(\mc{O}'',\mc{F})$ up to some shift. 	\end{theorem}
\begin{proof}
Since $\Omega$ is good then there exists $g_0\in \Omega$ such that $P'$ contains a Levi subgroup $\tM'$ of $Q=g_0Pg_0^{-1}$. 	Let $T'$ be a maximal torus inside $\tM'$. Since $T$ and $T'$ both are maximal inside $P'\cap Q$,  there exists $c\in P'\cap Q$, such that $T=cT'c^{-1}$. Then we can consider $c\tM'c^{-1}$ to be the new $\tM'$ which contains $T$ as maximal torus. This implies $\tM'=\tM$. Now $\tL$ and $Q$ both contain $T$ and $\tM$ is a Levi subgroup of $Q$ containing $T$. Further, $\tM$ is a reductive subgroup of $P'$. Hence, $\tM\subset \tL$ and $\tM$ is a Levi subgroup for the parabolic $\tL\cap Q$ of $\tL$. 

Consider the induction diagram for $\tL$ with parabolic $\tL\cap Q$ and Levi $\tM$ with the identification of $\tm$ with $\mf{m}$,
 \[
\begin{tikzcd}
\mc{O}''  & \mc{O}''+(\mf{l}_n'\cap\mf{v}_n) \arrow[l, "\pi"'] \arrow[r,  "e"] & \tL_0 \times^{\tL_0\cap Q_0}(\mc{O}''+(\mf{l}'_n\cap \mf{v}_n)) \arrow[r,"\mu"] & \tl_n,
\end{tikzcd}
\]note that $\tL_0 \times^{\tL_0\cap Q_0}(\mc{O}''+(\mf{l}'_n\cap \mf{v}_n))$ can be identified with
\[\{(l(\tL_0\cap Q_0), \zeta)\in \tL_0/(\tL_0\cap Q_0)\times \mf{l}'_n|\hspace{2mm} \mathrm{Ad}(l^{-1})\zeta\in \mc{O}''+(\mf{l}'_n\cap \mf{v}_n))\}.\]

If $(h(P'_0\cap Q_0),x) $ belongs to $\pno$, this implies $\mathrm{Ad}(h^{-1})x\in \mc{O}''+\mf{v}_n$ by (\ref{eq3.2}). But also $h\in P'_0$ and $x\in \mf{p}'_n$ implies $\mathrm{Ad}(h^{-1})x \in \mf{p}'_n$. Therefore \begin{equation}\label{eq3.3}
	\mathrm{Ad}(h^{-1})x \in  (\mc{O}''+\mf{v}_n )\cap \mf{p}'_n.
\end{equation}

Now we define a map \[d:\pno \to \tL_0 \times^{\tL_0\cap Q_0}(\mc{O}''+(\mf{l}'_n\cap \mf{v}_n)) \] by
$d (h(P'_0\cap Q_0),x)= (l(\tL_0\cap Q_0), \zeta)$, where $l$ is the image of $h$ under the projection $P'_0\to \tL_0$ and $\zeta $ is the image of $x$ under the projection $\mf{p}'_n\to \mf{l}'_n$. This map is well-defined since $\mathrm{Ad}(l^{-1})\zeta=\pi'(\mathrm{Ad}(h^{-1})x)\in \mc{O}''+(\mf{l}'_n\cap \mf{v}_n)$. We claim that this is a vector bundle. 
By (\ref{eq3.2}) and (\ref{eq3.3}),   $\pno$ is identified with  $P'_0\times^{P'_0\cap Q_0} (\mc{O}''+\mf{v}_n\cap \mf{p}'_n)$. Then we define a vector bundle map $P'_0\times^{P'_0\cap Q_0} (\mc{O}''+\mf{v}_n\cap \mf{p}'_n)\to P'_0\times^{P'_0\cap Q_0} (\mc{O}''+\mf{v}_n\cap \mf{l}'_n)$,   induced by the linear map  $\pi':\mf{p}'_n\to \mf{l}'_n$ with kernel $\mf{u}_n\cap \mf{v}'_n$.  Now consider the quotient map $P'_0\times^{P'_0\cap Q_0} (\mc{O}''+\mf{v}_n\cap \mf{l}'_n )\to P'_0\times^{U'_0(P'_0\cap Q_0)} (\mc{O}''+\mf{v}_n\cap \mf{l}'_n )$, whose fiber is ${U'_0(P'_0\cap Q_0)} /(P'_0\cap Q_0)$ which is isomorphic to $U'_0/(U'_0\cap P'_0\cap Q_0)\cong U'_0/(U'_0\cap Q_0)$, which is isomorphic to a vector space. Now the space $P'_0\times^{U'_0(P'_0\cap Q_0)} (\mc{O}''+\mf{v}_n\cap \mf{l}'_n ) $ is isomorphic to $L'_0\times^{(L'_0\cap Q_0)} (\mc{O}''+\mf{v}_n\cap \mf{l}'_n ) $. The map $d$ is the composition of all these maps,
\[\begin{tikzcd}
	P'_0\times^{P'_0\cap Q_0} (\mc{O}''+\mf{v}_n\cap \mf{p}'_n) \arrow[d, "\pi'"]\\  P'_0\times^{P'_0\cap Q_0} (\mc{O}''+\mf{v}_n\cap \mf{l}'_n) \arrow[d] \\ P'_0\times^{U'_0(P'_0\cap Q_0)} (\mc{O}''+\mf{v}_n\cap \mf{l}'_n ) \cong L'_0\times^{(L'_0\cap Q_0)} (\mc{O}''+\mf{v}_n\cap \mf{l}'_n )
\end{tikzcd}
,\] where the first and the second maps are vector bundles defined above. 
Therefore $d$ is a vector bundle of rank $\dim (\mf{u}'_n\cap \mf{v}_n) +\dim   U'_0/(U'_0\cap Q_0) $.
Consider the diagram below,
 \[\begin{tikzcd}
 	\mc{O}' & \mc{O}'+ \mf{u}_n \arrow[l, "\pi"] \arrow[r, hook, "e"] & G_0\times^{P_0} (\mc{O}'+\mf{u}_n) \\
 	\mc{O}''\arrow[u, "\mathrm{Ad}(g_0^{-1})"]& \mc{O}''+\mf{v}_n \arrow[l, "\pi'"]\arrow[u, "\mathrm{Ad}(g_0^{-1})"] \arrow[r, hook, "e'"']  & G_0\times^{Q_0}(\mc{O}''+\mf{v}_n) \arrow[u, "\mathrm{Ad}(g_0^{-1})"]
 	\\
 	& \mc{O}'' +(\mf{p}'_n\cap \mf{v}_n) \arrow[u, hook', "k"]  \arrow[r, hook, "f"] \arrow[d, "d'"]&  P'_0 \times^{P'_0\cap Q_0}(\mc{O}''+(\mf{p}'_n\cap \mf{v}_n))  \arrow[rrd, "\sigma_\Omega"]\arrow[u, hook', "j"] \arrow[d, "d"]\arrow[uu, bend right=75, hook, "i'"]\\ 
 	& \mc{O}'' +(\mf{l}'_n\cap \mf{v}_n) \arrow[uu, bend left=70, hook', "k'"] \arrow[r, hook, "e_{L'}"]&  L'_0 \times^{L'_0\cap Q_0}(\mc{O}''+(\mf{l}'_n\cap \mf{v}_n)) \arrow[rr, "\mu_{L'}"] && \mf{l}_n'  
 \end{tikzcd}.\]
 We are trying to calculate $K_\Omega$, which is by definition $\sigma_{{\Omega}_!} i'^*\f^{G_0}_{P_0'}(e^*\f^{G_0}_{P_0})^{-1}\pi^*\mc{L}'[\dim \mc{O}']$.  
 
 From the first two rows of the above diagram,
 $\mathrm{Ad}(g_0^{-1})^*(e^*\f^{G_0}_{P_0})^{-1}\pi^*\mc{L}'[\dim \mc{O}'] \cong (e'^*\f^{G_0}_{Q_0})^{-1}\pi'^*\mathrm{Ad}(g_0^{-1})^*\mc{L}'[\dim \mc{O}'] $. Let us call $\mathrm{Ad}(g_0^{-1})^*\mc{L}'[\dim \mc{O}'] $ as $\mc{L}''$ and we get,
 
 \begin{equation}\label{eq3.5}
 	i'^*\f^{G_0}_{P_0'}(e^*\f^{G_0}_{P_0})^{-1}\pi^*\mc{L}'[\dim \mc{O}']\cong j^*\f^{G_0}_{P_0'} (e'^*\f^{G_0}_{Q_0})^{-1}\pi'^* \mc{L}''.
 \end{equation}
 Now from the commutative square involving $e',j,k,f$ and the formalism of the  equivariant derived category we get,
 \[ k^*\f^{Q_0}_{P_0'\cap Q_0}e'^* \f^{G_0}_{Q_0}\cong f^* \f^{P_0'}_{P_0'\cap Q_0}j^* \f^{G_0}_{P_0'}. 
 \] As  $ e'^* \f^{G_0}_{Q_0}$ and $f^* \f^{P_0'}_{P_0'\cap Q_0}$ are equivalences of categories, we have
 
 \begin{equation}\label{eq3.6}
 	j^*\f^{G_0}_{P_0'} (e'^*\f^{G_0}_{Q_0})^{-1}\pi'^* \mc{L}''\cong (f^*\f^{P_0'}_{P'_0\cap Q_0})^{-1}k^*\f^{Q_0}_{P'_0\cap Q_0}\pi'^* \mc{L}''.
 \end{equation}
 
 Using the fact that $\sigma_\Omega=\mu_{L'}\circ d$, (\ref{eq3.5}) and (\ref{eq3.6}), we see that, 
 \begin{equation}\label{eq3.7}
 K_\Omega\cong	\mu_{{L'}_!} d_! (f^*\f^{P_0'}_{P'_0\cap Q_0})^{-1}k^*\f^{Q_0}_{P'_0\cap Q_0}\pi'^* \mc{L}''.
 \end{equation} 
 
 Now consider the commutative diagram,
 \[\begin{tikzcd}
 \mc{O}'' &\mc{O}''+\mf{v}_n \arrow[l, "\pi'"]\\	
 \mc{O}''+(\mf{l}'_n\cap \mf{v}_n) \arrow[u, "\pi''"] & \mc{O}''+(\mf{p}'_n\cap \mf{v}_n) \arrow[l, "d'" ]\arrow[u, hook', "k"].
 \end{tikzcd}\]
From this we get, 
\[d'^*\pi''^*\f^{Q_0}_{P'_0\cap Q_0}\cong k^* \f^{Q_0}_{P'_0\cap Q_0}\pi'^*.\]

 Therefore (\ref{eq3.7}) becomes,

 \begin{equation}\label{eq3.8}
 	K_{\Omega}\cong \mu_{{L'}_!} d_! (f^*\f^{P_0'}_{P'_0\cap Q_0})^{-1}d'^*\pi''^*\f^{Q_0}_{P'_0\cap Q_0}\mc{L}''.
 \end{equation} 
By the formalism of the equivariant derived category from the square associating $d,d',f, e_{L'}$,

\[d'^* e_{L'}^* \f^{L_0'}_{L_0'\cap Q_0}\cong f^* \f^{P'_0}_{P_0'\cap Q_0}d^*.\] This implies the sheaf complex in (\ref{eq3.8}) becomes,

\begin{equation}\label{eq3.9}
K_\Omega\cong 	\mu_{{L'}_!} d_! d^*(e_{L'}^* \f^{L_0'}_{L_0'\cap Q_0})^{-1} \pi''^*\f^{Q_0}_{P'_0\cap Q_0} \mc{L}''.
 \end{equation}

 As $d$ is a vector bundle, we have,

 \[d_! d^*(e_{L'}^* \f^{L_0'}_{L_0'\cap Q_0})^{-1} \pi''^*\f^{Q_0}_{P'_0\cap Q_0}\mc{L}''\cong (e_{L'}^* \f^{L_0'}_{L_0'\cap Q_0})^{-1} \pi''^*\f^{Q_0}_{P'_0\cap Q_0}\mc{L}''[2m],
 \]
  where $m$ is the rank of d.
Therefore (\ref{eq3.9}) becomes,

\begin{equation}\label{eq3.10}
 K_\Omega\cong	\mu_{{L'}_!} (e_{L'}^* \f^{L_0'}_{L_0'\cap Q_0})^{-1} \pi''^*\f^{Q_0}_{P'_0\cap Q_0}\mc{L}''[2m].
 \end{equation} 

We can call $\f^{Q_0}_{P'_0\cap Q_0}\mc{L}''[2m] $ by $\mc{F}$, which is a $P_0'\cap Q_0$ equivariant local system on $\mc{O}''$ and (\ref{eq3.10}) is 
\[\inn^{\mf{l}'}_{\mf{l}'\cap \mf{q}} \mc{IC}(\mc{O}'',\mc{F}),\] and we are done.

\end{proof}

\subsection{Bad double cosets}In this section we assume $\Omega$ is bad.

\begin{theorem}\label{th2.3}
	For $\Omega$ bad, $K_{\Omega}=0$.
\end{theorem}

\begin{proof}
Assume $\Omega$ is bad. As in the good case,
 \[\pno= \{(h(P'_0\cap Q_0),x)\in P'_0/{P'_0\cap Q_0}\times \mf{p}'_n|\mathrm{Ad}(h^{-1})x\in \mc{O}''+\mf{v}_n\},\]	and $\sigma_\Omega: \pno\to \mf{l}_n'$ is defined by $(h(P_0'\cap Q_0), x)\to \pi'(x)$. Let $y\in \mf{l}_n'$ and we want to show, $(K_\Omega)_y=0$.  Consider the following Cartesian diagram, 
\[\begin{tikzcd}
\pno \arrow[r, "\sigma_\Omega"'] & \mf{l}_n'\\ {\sigma_\Omega}^{-1}(y) \arrow[u, hook'] \arrow[r, "\sigma_\Omega"] & y	\arrow[u, hook']
\end{tikzcd}.\]
To show  $(K_\Omega)_y=0$, using the above diagram and \cite[Lemma ~11]{Ch}, it is enough to show 
\[R\Gamma_c(((e^*\f^{G_0}_{P_0})^{-1}\pi^*\mc{L}'[\dim \mc{O}']|_{{\sigma_\Omega}^{-1}(y)}) )=0.\]Which is identified with, 
\[\h^*_{c}({\sigma_\Omega}^{-1}(y), ((e^*\f^{G_0}_{P_0})^{-1}\pi^*\mc{L}'[\dim \mc{O}']|_{\pno}) ).\]
For an algebraic variety $X$, let $p_X$ be the map from $X$ to a point. Let  $\mc{F}$ be $(e^*\f^{G_0}_{P_0})^{-1}\pi^*\mc{L}'[\dim \mc{O}']|_{\sigma_\Omega^{-1}(y)} $.  We must show $p_{{\sigma_\Omega^{-1}(y)}_!}\mc{F}=0$.

Consider the diagram below for $h\in P'_0$, 
\[\begin{tikzcd}
\pr_1^{-1}(h(P'_0\cap Q_0)) \arrow[r, hook, "j_h"] \arrow[d, ]& \sigma_\Omega^{-1}(y)\arrow[d, "\pr_1"] \arrow[r] & y\\ \{h(P'_0\cap Q_0)\} 	\arrow[r,hook, "i_h"] & P'_0/(P'_0\cap Q_0) \arrow[r] & y
\arrow[u, "="]\end{tikzcd},
\] where $\pr_1$ is the projection on the first coordinate. 

Since $p_{{\sigma_\Omega^{-1}(y)}_!}\mc{F}= p_{{P'_0/(P'_0\cap Q_0)}_!}{\pr_1}_!\mc{F} $, it suffices to show ${\pr_1}_!(\mc{F})=0$.
From the above diagram, it is enough to show, $i_h^* {\pr_1}_!(\mc{F})=0$ for each $h\in P_0'$.
By base change,
\begin{align}\label{eq3.4}
	i_h^* {\pr_1}_!(\mc{F})={p_{\pr_1^{-1}(h(P'_0\cap Q_0))}}_!j_h^*\mc{F},
\end{align}
where $\pr_1^{-1}(h(P'_0\cap Q_0))=\{z\in \mf{p}'_n| \hspace{2mm} \mathrm{Ad}(h^{-1})z\in \mc{O}''+\mf{v}_n, \pi'(z)=y\}$. Now $\mathrm{Ad}(h^{-1})z$ also belongs to $\mf{p}'_n$ as $h\in P'_0$ and $z\in \mf{p}'_n$. So $\mathrm{Ad}(h^{-1})z\in \mf{p}'_n\cap \mf{q}_n$. Since $P'$ and $Q$ both contain $T$, the intersection $\mf{p}'_n\cap \mf{q}_n$ can be written as the following direct sums,
\begin{equation}\label{eq3.12}
	\mf{p}'_n\cap \mf{q}_n=(\mf{l}'_n\cap \mf{m}_n)\oplus (\mf{u}'_n\cap \mf{m}_n) \oplus (\mf{p}'_n\cap \mf{v}_n),\hspace{.1cm}
 (\mf{l}'_n\cap \mf{m}_n)\oplus (\mf{l}'_n\cap \mf{v}_n)\oplus  (\mf{u}'_n\cap \mf{q}_n).
\end{equation}

 Using the first decomposition $\mathrm{Ad}(h^{-1})z$ can be written as $\gamma+ \nu+\mu$ and using the second decomposition $\gamma+ \nu'+\mu'$, where $\gamma\in (\mf{l}'_n\cap \mf{m}_n) $, $\nu\in (\mf{u}'_n\cap \mf{m}_n) $, $\mu\in (\mf{p}'_n\cap \mf{v}_n) $, $\nu'\in (\mf{l}'_n\cap \mf{v}_n) $ and $\mu'\in (\mf{u}'_n\cap \mf{q}_n) $. From the fact
\[\gamma+\nu+\mu=\gamma+\nu'+\mu',\]
we get $\mu'-\nu=\mu-\nu'\in (\mf{u}'_n\cap \mf{v}_n) $, call it $\tilde{\mu}$. Then $\mathrm{Ad}(h^{-1})z=\gamma+\nu+\nu'+\tilde{\mu}$. Now since $\gamma+\nu\in (\mf{p}'_n\cap \mf{m}_n) $,  $\gamma+\nu\in \mc{O}''$. As $\pi'(z)=y\in \mf{l}'_n$ is fixed, therefore $\gamma$ and $\nu'$ are being uniquely determined by the condition $\gamma+\nu'\in \mf{l}_n'$ and $\mathrm{Ad}(h^{-1})z-(\gamma+\nu')\in \mf{u}'_n$. So, 
\begin{align}\label{eq3.1}
	\pr_1^{-1}(h(P'_0\cap Q_0))\cong \mc{O}''\cap(\gamma+(\mf{u}'_n\cap \mf{m}_n ))\times (\mf{u}'_n\cap \mf{v}_n).
\end{align}
 Now consider the following diagram,
\[\begin{tikzcd}
\mc{O}' & \mc{O}'+\mf{u}'_n \arrow[l, "\pi"']\arrow[r, hook] & G_0\times^{P_0} (\mc{O}'+\mf{u}_n) & 	\\
\mc{O}'' \arrow[u, "Ad(g_0^{-1})"] && \sigma_\Omega^{-1}(y) \arrow[u,hook'] & \\\mc{O}''\cap(\gamma+(\mf{u}'_n\cap \mf{m}_n )) \arrow[u, hook] \arrow[rr, hook]\arrow[d, "\pr_2"] & &\mc{O}''\cap(\gamma+(\mf{u}'_n\cap \mf{m}_n ))\times (\mf{u}'_n\cap \mf{v}_n) \arrow[u, hook'] \arrow[d, "\pr_2"]\\ \{pt\} \arrow[rr, hook] && \mf{u}'_n\cap \mf{v}_n 
\end{tikzcd}
\]
From  (\ref{eq3.1}), proving
\[ R\Gamma_c(\mc{F}|_{\pr_1^{-1}(h(P'_0\cap Q_0))}) )=0\]is the same as proving,
\[R\Gamma_c(((e^*\f^{G_0}_{P_0})^{-1}\pi^*\mc{L}'[\dim \mc{O}']|_{\mc{O}''\cap(\gamma+(\mf{u}'_n\cap \mf{m}_n )\times(\mf{u}'_n\cap \mf{v}_n) )}) )=0.\] Which is from the above diagram is same as proving, 
\[R\Gamma_c(\mc{L}''|_{ \mc{O}''\cap(\gamma+(\mf{u}'_n\cap \mf{m}_n ))})=0.\] But the last term is isomorphic to the stalk of  \[\rnn^{\mf{m}}_{\mf{p}'\cap \mf{m}}\mc{E}(\mc{O}'',\mc{L}'').\] Which is indeed $0$ since $(\mc{O}'',\mc{L}'')\in \ms{I}(\mf{m}_n)^{\cu}$.
\end{proof}

\subsection{Composition of $\inn^{\mf{g}}_{\mf{p}}$ and $\rnn^{\mf{g}}_{\mf{p}}$} In this subsection we prove the following theorem about the composition of $\rnn^{\mg}_{\mf{p}'}$ and $\inn^{\mg}_{\mf{p}}$. 
\begin{theorem}\label{Th3.4} For any pair $(\mc{O}',\mc{L}')\in \ms{I}(\mf{l}_n)$, 
	$\rnn^{\mg}_{\mf{p}'}\inn^{\mg}_{\mf{p}}\mc{E}(\mc{O}',\mc{L}')$ is a parity complex.
\end{theorem}
\begin{proof}
 From (\ref{eq3.2'}), 	\[\rnn^{\mg}_{\mf{p}'}\inn^{\mg}_{\mf{p}}\mc{E}(\mc{O}',\mc{L}')\cong \sigma_!((e^*\f^{G_0}_{P_0})^{-1}\pi^*\mc{L}'[\dim \mc{O}']|_{\mu^{-1}(\mf{p}'_n)}).\]
 For a $(P_0',P_0)$-stable subset $\Omega$ of $G_0$, recall $\pno$ from the beginning of section \ref{sec2}. We can choose  $(P'_0,P_0)$-double cosets $\Omega_0,\dots, \Omega_m$, such that $\mu^{-1}(\mf{p}_n')=\cup_{i=0}^m\pno_i$ and $\sigma_{\Omega_i}:\pno_i\rightarrow \mf{l}_n'$ is defined by $\sigma|_{\Omega_i}$. We assume for any $j\in [0,m]$,  $\pno_0\cup \dots, \cup \pno_j$ is closed in $\pno$. Let $\sigma_j$ be the map   $\pno_0\cup \dots, \cup \pno_j\rightarrow \mf{l}'_n$ defined by various $\sigma_{\Omega_i}$. By induction we will show for any $j$,  \[{\sigma_j}_!((e^*\f^{G_0}_{P_0})^{-1}\pi^*\mc{L}'[\dim \mc{O}']|_{\pno_0\cup \dots, \cup \pno_j}) \] is a parity complex. This will imply, \[{\sigma_m}_!((e^*\f^{G_0}_{P_0})^{-1}\pi^*\mc{L}'[\dim \mc{O}']|_{\mu^{-1}(\mf{p}'_n)}) \cong \sigma_!((e^*\f^{G_0}_{P_0})^{-1}\pi^*\mc{L}'[\dim \mc{O}']|_{\mu^{-1}(\mf{p}'_n)}) \] is a parity complex. Consider the maps
\[
	 \pno_0\cup \dots, \cup \pno_j \xhookrightarrow{i}  \pno_0\cup \dots, \cup \pno_{j+1}  \xhookleftarrow{j} \pno_{j+1},
\] with $i$ a closed embedding and $j$ an open embedding. This induces
distinguished triangle,
\[ i_!i^* {\sigma_{j+1}}_!((e^*\f^{G_0}_{P_0})^{-1}\pi^*\mc{L}'[\dim \mc{O}']|_{\mu^{-1}(\mf{p}'_n)}) \to {\sigma_{j+1}}_!((e^*\f^{G_0}_{P_0})^{-1}\pi^*\mc{L}'[\dim \mc{O}']|_{\mu^{-1}(\mf{p}'_n)}) \to j_* K_{\Omega_j}\to 
\]Now the first term is a parity complex by induction and third term is a  parity complex by Theorem \ref{th2.2} and \cite[Theorem ~37]{Ch}. Therefore the middle term is a parity complex.
\end{proof}
\begin{cor}
	For any parabolic subgroup $P$ with Levi $L$ containing $\chi(\mb{C}^\times)$ , $\rnn^\mf{g}_\mf{p}$ sends parity complexes to parity complexes.
\end{cor}

\begin{proof}
	We know that for any pair $(\mc{O},\mc{L})\in \i'(\mf{g}_n)$, there exists a parabolic subgroup $P'$ of $G$ with Levi $L'$ and $(\mc{O}_L,\mc{E}_L)\in \i'(\mf{l}'_n)^{\cu}$ such that $\mc{E} (\mc{O},\mc{L}) $ appears as direct summand of $\inn^{\mf{g}}_{\mf{p}'}\mc{E} (\mc{O}_L,\mc{E}_L) $ \cite{Ch}. Therefore $ \rnn^\mf{g}_\mf{p} \mc{E} (\mc{O},\mc{L}) $ appears as a direct summand of $ \rnn^\mf{g}_\mf{p}\inn^{\mf{g}}_{\mf{p}'}\mc{E} (\mc{O}_L,\mc{E}_L)      $, which is a parity complex by Theorem \ref{Th3.4}. Hence $ \rnn^\mf{g}_\mf{p} \mc{E} (\mc{O},\mc{L}) $  is parity complex.
\end{proof}

\section{Fourier Sato transform}\label{sec4}

Fourier transform was studied by Hotta-Kashiwara \cite{HK} for $\mc{D}$-modules and later by Brylinski \cite{Bry} in characteristic $0$. Study of Fourier transform in the context of the Springer sheaf in positive characteristic has been initiated by  Juteau \cite{Ju}. 

  In this section we will recall the Fourier-Sato transform \cite{AJHR}.  Being a $t$-exact functor, it is obvious that Fourier transform sends semi-simple complexes to semi-simple complexes. Here we prove that in positive characteristic it sends parity complexes to parity complexes. In \cite{Lu}
 it has been proved that it sends the cuspidal sheaves to cuspidal sheaves in characteristic $0$. We prove here the same statement but in positive characteristic. Our proof is much  simpler than the proof in \cite{Lu}.

\subsection{Fourier transform on vector spaces}

Let $V$ be a  complex vector space and $H$ be an algebraic group acting linearly on $V$. An object $\mc{F}$ in the $H$-equivariant bounded derived category of $V$, $D^b_H(V,\Bbbk)$ called conic if for any $v\in V$ and $i\in \mb{Z}$, the sheaf $\h^i(\mc{F})|_{{\mb{C}.v}-0}$ is locally constant, Or equivalently if $\mc{F}$ is contructible with respect to the $\mb{C}^\times$-orbits in $V$.  Let us denote by $D^b_{H,con}(V,\Bbbk)$ the full subcategory of conic objects in $D^b_H(V,\Bbbk)$.
We denote the dual of $V$ by $V^*$. 

The Fourier transform is a functor,
\[\Phi_V: D^b_{H,con}(V,\Bbbk)\to D^b_{H,con}(V^*,\Bbbk).\]
 
 This functor was initially introduced in \cite[ ~3.7]{KS} and modified in \cite{AM} with a shift of $[\dim V]$ so that the functor is $t$-exact for the perverse $t$-structure \cite[Prop. ~10.3.8]{KS}. This functor is an equivalence of categories with the following properties:
 \begin{enumerate}
 	\item $\Phi_V(\delta_V)=\Bbbk_{V^*}[\dim V^*]$,  where $\delta_V$ is the skyscraper sheaf at $\{0\}$ on $V$. 
 	\item It commutes with the external tensor product, that is for $V=V_1\times V_2$, 
 	 \[ \Phi_V(M_1\boxtimes M_2)\cong \Phi_{V_1} (M_1)\boxtimes \Phi_{V_2}(M_2) ,
 	 \] with $M_1\in D^b_{H,con }(V_1)$ and $M_2\in D^b_{H,con}(V_2)$.
 \end{enumerate}

\subsection{Fourier transform on graded Lie algebras}
We consider the Fourier-Sato transform on the vector space $\mf{g}_n$.  Since  $G_0$-orbits in $\mg_n $ are $\mb{C}^\times$ stable, $D^b_{G_0,con}(\mf{g}_n) =D^b_{G_0}(\mf{g}_n)$. Since $\mg_n^*$ is $G_0$-isomorphic to $\mg_{-n}$, we obtain
\[\phign: D^b_{G_0}(\mf{g}_n)\to D^b_{G_0}(\mf{g}_{-n}). 
\]
Here we assume the parabolic subgroup $P$ and its Levi subgroup $L$ contain the image $\chi(\mb{C}^\times)$.
In this section we will use the notations,
 \[\ind_n : D^b_{L_0}(\mf{l}_n) \to D^b_{G_0}(\mf{g}_n)
 \]
and \[ \res_n: D^b_{G_0}(\mf{g}_n)\to D^b_{L_0}(\mf{l}_n),
\] for induction and restriction respectively.

In \cite{Lu} the commutativity of the Fourier transform with the induction and restriction has been proved.  Here we give a simpler proof using the proper base change and duality. 
\begin{theorem}\label{th3.1}
	For a parabolic subgroup $P\subset G$ with $\chi(\mb{C}^\times)\subset P$,
	$\Phi_{\mf{l}_n} \res_{n}\cong \res_{-n}\phign$.
\end{theorem}

\begin{proof}
Let $\bar{\mf{p}}$ be the opposite Lie algebra of $\mf{p}$, which means $\bar{\mf{p}}=\mf{l}\oplus \bar{\mf{u}}$. Consider the following Cartesian diagram from \cite{Mi}, 
\[\begin{tikzcd}
\mf{p}_n \arrow[r, "\pi"] \arrow[d, hook', "i"]	& \mf{l}_n \arrow[d, hook', "j"]\\ \mf{g}_n \arrow[r, "\tau"] & \mf{\bar{p}}_n
\end{tikzcd},
\]
where $j$ is inclusion and $\tau$ is projection from $\mg_n$ to $\bar{\mf{p}}_n$. Therefore 
for $\mc{F}\in D^b_{G_0}(\mf{g}_n)$ we have,
\[ \res_n(\mathcal{F})= {\pi}_!i^* \f^{G_0}_{L_0}(\mathcal{F})\cong  j^!\tau_* \f^{G_0}_{L_0}(\mathcal{F}).
\]The Lie algebras, $\mf{p}_n$ and $\bar{\mf{p}}_{-n}$ are dual to each other, and the inclusion $i: \mf{p}_n\hookrightarrow \mf{g}_n$ has dual map $\tau: \mf{g}_{-n} \to \bar{\mf{p}}_{-n}$. Hence by \cite[prop. ~6.9.13]{Ac}, 
\begin{equation}\label{eq4.14}
	{\tau}_* \phign  \mathcal{F}\cong \Phi_{\mf{p}_n}i^! \mathcal{F}[\dim \mf{g}_n-\dim \mf{p}_n].
\end{equation}
Similarly, the map $j:\mf{l}_{-n}\to \bar{\mf{p}}_{-n}$ is dual to the map $\pi: \mf{p}_n\to \mf{l}_n$. Hence by the similar argument as above we have,
\begin{equation}\label{eq4.15}
	j^! \Phi_{\mf{p}_n}\cong \Phi_{\mf{l}_n} \pi_*[\dim \mf{l}_n-\dim \mf{p}_n ] .
\end{equation}

Therefore we get,
\begin{align*}
\res_{-n}	\Phi_{\mf{g}_n}\mc{F} &= j^!\tau_* \f^{G_0}_{L_0}  \Phi_{\mf{g}_n}  \mc{F} \\&
 \cong   j^! \Phi_{\mf{p}_n}i^! \f^{G_0}_{L_0}\mc{F}[\dim \mf{g}_n -\dim \mf{p}_n] \text{,  by } (\ref{eq4.14}) \\ & \cong  \Phi_{\mf{l}_n} \pi_*i^! \f^{G_0}_{L_0} \mc{F}[\dim \mf{l}_n +\dim \mf{g}_n-2\dim \mf{p}_n]  \text{, by }(\ref{eq4.15})\\ & \cong \Phi_{\mf{l}_n} \res_n \mc{F}.
\end{align*}

\end{proof}

\begin{cor}\label{cor1.2}
	For $P\subset G$ with $\chi(\mb{C}^\times)\subset P$,
	$\phign \ind_{n}\cong \ind_{-n}\Phi_{\mf{l}_n}$.
\end{cor}

\begin{proof}
	The proof follows from the fact that the restriction is left adjoint to the induction.
\end{proof}

Let $(\mc{O},\mc{L})\in \ms{I}(\mf{g}_n)$. By Theorem 36 of \cite{Ch}, there exists a parabolic subgroup $P$ with Levi subgroup $L$ containing $\chi(\mb{C}^\times)$ and a cuspidal pair $(\mc{O}_L,\mc{E}_L)\in \ms{I}(\mf{l}_n)^{\cu}$ such that $\mc{E}(\mc{O},\mc{L})$ appears as direct summand of $\ind_{n}\mc{IC}(\mc{O}_L,\mc{E}_L)$.

\begin{lemma}\label{lm4.3}
	Let $(\mc{O},\mc{L})\in \ms{I}(\mf{g}_n)$ and $P$, $L$ and $(\mc{O}_L,\mc{E}_L)\in \ms{I}(\mf{l}_n)^{\cu}$ be the associated parabolic, Levi subgroups and the cuspidal pair. The pair $(\mc{O},\mc{L})$ is cuspidal if and only if $L=G$.
\end{lemma}

\begin{proof}
	If $L=G$, then it is obvious that $(\mc{O},\mc{L})$ is cuspidal. On the other hand let $(\mc{O},\mc{L})$ be cuspidal. But by \cite[Theorem ~28]{Ch}, there exists a $L_0$-orbit $\mc{O}'\subset \mf{l}_n$ such that $\mc{O}'\subset \mc{O}$. This implies $\mc{O}\cap \mf{l}_n\ne \emptyset$ for a proper Levi subgroup $L$. But $(\mc{O},\mc{L})$ cuspidal means there exists a cuspidal pair $(C,\mc{F})\in \ms{I}(G)$ such that $\mc{O}\subset C$. This shows that $C\cap \mf{l}\ne \emptyset$, which contradicts the fact that $C$ is distinguished in $\mf{g}$.
\end{proof}

\begin{theorem}
	Let $(\mc{O},\mc{L})\in \ms{I}(\mf{g}_n)$ be a cuspidal pair. Then for any proper parabolic $P$, $\res_{n}\mc{IC}(\mc{O},\mc{L})=0$.
\end{theorem}
\begin{proof}
	By the definition of cuspidal pair for the graded setting, there exists a cuspidal $(C,\mc{E})\in \i'(G)^{\cu}$ such that $\mc{O}=C\cap \mf{g}_n$ and $\mc{L}=\mc{E}|_{\mc{O}}$. Let $P$ be a proper parabolic subgroup with Levi subgroup $L$ containing $\chi(\mb{C}^\times)$. Let $U$ be the unipotent radical of $P$. Let $x\in \mf{l}_n$. We want to show $(\res_{n} \mc{IC}(\mc{O},\mc{L}))_x=0$. Using the base change, 
\begin{equation*}
	(\res_n \mc{IC}(\mc{O},\mc{L}))_x=R\Gamma_c(\f^{G_0}_{L_0}\mc{L}|_{(x+\mf{u}_n)\cap \mc{O}}).
\end{equation*}  Now by the definition of cuspidal in the non-graded setting we have, $(\rnn^G_P \mc{IC}(C,\mc{E}))_x=0$ for any proper parabolic subgroup $P$ of $G$,  or equivalently 
\begin{equation}\label{eq4.7}
 	R\Gamma_c(\f^{G}_{L}\mc{E}|_{(x+\mf{u})\cap C}) =0.
 \end{equation}
By \cite[Lemma ~3.1]{Ch}, $R\Gamma_c(\mc{E}|_{(x+\mf{u})\cap C}) =0$. Consider the action of $\mb{C}^\times$ on $\mf{g}$, defined by 
$a.y=a^{-n} Ad(\chi(a))y$ and consider the variety $(x+\mf{u})\cap C$. The fixed point set inside this variety is $(x+\mf{u}_n)\cap \mc{O}$. Now using (\ref{eq4.7}) and \cite[Lemma ~3.5]{Ch} we get,
 \begin{equation*}
 	R\Gamma_c(\mc{L}|_{(x+\mf{u_n})\cap \mc{O} })=0.
 \end{equation*}
 From \cite[Lemma ~3.1]{Ch} It follows that,
 \begin{equation*}
 	R\Gamma_c(\f^{G_0}_{L_0}\mc{L}|_{(x+\mf{u_n})\cap \mc{O} })=0,
 \end{equation*}which proves the theorem.
\end{proof}

\begin{theorem}\label{th1.3}
\begin{enumerate}
	\item Let $(\mc{O},\mc{L})\in \ms{I}(\mf{g}_n)^{\cu}$. Then $\phign(\mc{IC}(\mc{O},\mc{L}))=\mc{IC}(\mc{O}',\mc{L}')$ for some $(\mc{O}',\mc{L}')\in \ms{I}(\mf{g}_{-n})^{\cu}$.
	\item $\phign$ sends parity sheaves to parity sheaves.
\end{enumerate}
	
\end{theorem}

\begin{proof}The functor 
$\phign: D^b_{G_0}(\mf{g}_n)\to D^b_{G_0}(\mf{g}_{-n})$ sends simple perverse sheaves to simple perverse sheaves. Now  our claim is that it sends non-cuspidal pairs to non-cuspidal pairs, which suffices to prove  that $\phign$ sends cuspidal pairs to cuspidal pairs as it is an equivalence of categories. 
We will prove this by induction on the semisimple rank of $G$. 
For the base case, when the semisimple rank of $G$ is $0$, it is a torus. Hence there is nothing to prove.
Now we want to show that  for the non-cuspidal pair $(\mc{O},\mc{L})\in\ms{I}(\mf{g}_n)$,  $\phign(\mc{IC}(\mc{O},\mc{L}))$ is non-cuspidal. Now for a while we pause this proof and prove (2) for any non-cuspidal pair.


Consider the non-cuspidal pair $(\mc{O},\mc{L})\in \i'(\mf{g}_n)$ and let $\mc{E}(\mc{O},\mc{L})$ be the associated parity sheaf. By   \cite[Theorem ~36]{Ch}, there exists a parabolic subgroup $P$ with a Levi $L$ containing $\chi(\mb{C}^\times)$ and  	$(\mc{O}_L,\mc{E}_L) \in \ms{I}(\mf{l}_n)^{\cu}$ such that $ \mc{E}(\mc{O},\mc{L})$ is a direct summand of $\ind_{n}\mc{IC}(\mc{O}_L,\mc{E}_L)$. By Lemma \ref{lm4.3}, $P$ is proper. By Corollary \ref{cor1.2}, $ \phign\mc{E}(\mc{O},\mc{L})$ is direct summand of $\ind_{-n} \Phi_{\mf{l}_n}\mc{IC}(\mc{O}_L,\mc{E}_L)$. By induction, $\Phi_{\mf{l}_n}\mc{IC}(\mc{O}_L,\mc{E}_L)$ is cuspidal. This implies that $\ind_{n} \Phi_{\mf{l}_n}\mc{IC}(\mc{O}_L,\mc{E}_L)$ is parity by \cite[Theorem ~26]{Ch}. Therefore  $ \phign\mc{E}(\mc{O},\mc{L})$ is a parity complex.  Also as $\phign$ is an equivalence of categories, it sends indecomposable objects to indecomposable objects. Hence $\phign \mc{E}(\mc{O},\mc{L})$ is an indecomposable parity complex. This proves assertion (2) for non-cuspidal pairs. By the cleanness property of cuspidal pairs \cite[Theorem ~22]{Ch}, if $(\mc{O},\mc{L})$ is cuspidal then $\mc{IC}(\mc{O},\mc{L})=\mc{E}(\mc{O},\mc{L})$. So (2) will follow from (1).

Now we come back to the proof of (1) for the non-cuspidal pair $(\mc{O},\mc{L})$. We know $\mc{IC}(\mc{O},\mc{L})$ appears as a composition factor of $\mc{E}(\mc{O},\mc{L})$. So $\phign\mc{IC}(\mc{O},\mc{L})$ appears as a composition factor of $\phign \mc{E}(\mc{O},\mc{L})$. As we already know $\phign$ sends parity sheaves to parity sheaves for a non-cuspidal pair,  $\phign \mc{E}(\mc{O},\mc{L})$ is a parity sheaf, say $ \mc{E}(\mc{O}',\mc{L}')$. As $\phign$ sends simple perverse sheaves to simple perverse sheaves, hence $\phign\mc{IC}(\mc{O},\mc{L})=\mc{IC}(\mc{O}'',\mc{L}'')$, for some $(\mc{O}'',\mc{L}'')\in \ms{I}(\mf{g}_n)$. Our goal is to prove $(\mc{O}'',\mc{L}'')$ is non-cuspidal. If $\mc{O}''$ is not open then by \cite[~4.4(a)]{Lu}, $(\mc{O}'',\mc{L}'')$ is non-cuspidal. Hence we may assume $\mc{O}''$ is open. But we also have $\mc{O}''\subset \bar{\mc{O}'}$ and this implies $\mc{O}''=\mc{O}'$. This forces $\mc{L}'=\mc{L}''$. As we already know $\phign(\mc{E}(\mc{O},\mc{L}))=\mc{E}(\mc{O}',\mc{L}')$ appears as direct summand of $\ind_{n}\Phi_{\mf{l}_n}\mc{IC}(\mc{O}_L,\mc{E}_L)$, where $P$ is a proper parabolic subgroup and $\Phi_{\mf{l}_n}\mc{IC}(\mc{O}_L,\mc{E}_L) $ is cuspidal. This means $\mc{IC}(\mc{O}',\mc{L}')$ appears as a direct summand of $\ind_{-n} \Phi_{\mf{l}_n}\mc{IC}(\mc{O}_L,\mc{E}_L)$. But this implies by adjunction, $\res_n\mc{IC}(\mc{O}',\mc{L}')\ne0$. Hence $(\mc{O}',\mc{L}')$ is non-cuspidal. This completes the proof of (1).

\end{proof}

\end{document}